\documentclass[11pt]{amsart}
\usepackage[left=1in,right=1in,top=1in,bottom=1in]{geometry}

\usepackage[all,cmtip]{xy}
\usepackage{amssymb}
\usepackage[english]{babel}
\usepackage{mathrsfs}

\newtheorem{theorem}{Theorem}[section]
\newtheorem{lemma}[theorem]{Lemma}
\newtheorem{fact}[theorem]{Fact}

\newtheorem{corollary}[theorem]{Corollary}

\theoremstyle{definition}
\newtheorem{definition}[theorem]{Definition}

\newtheorem{example}[theorem]{Example}
\newtheorem{problem}[theorem]{Problem}

\newtheorem{remark}[theorem]{Remark}

\newtheorem{question}[theorem]{Question}

\newtheorem{claim}{Claim}

\def\N{\mathbb{N}}

\def\Z{\mathbb{Z}}

\def\Z{\mathbb{Z}}

\def \Zet[#1]{\lceil #1 \rceil}

\def\wprod{*}

\usepackage{color}
\newenvironment{revds}{\color{red}}{}
\newenvironment{revvy}{\color{magenta}}{}

\newenvironment{revnew}{\color{blue}}{}

\def\bds{\begin{revds}}
\def\eds{\end{revds}}

\def\by{\begin{revvy}}
\def\ey{\end{revvy}}

\def\bn{\begin{revnew}}
\def\en{\end{revnew}}

\begin{document}
\title[Markov's problem
for free 
groups]
{Markov's problem
for free groups}

\author[D. Shakhmatov]{Dmitri Shakhmatov}
\address{Graduate School of Science and Engineering\\
Ehime University, Matsuyama 790-8577, Japan}
\email{dmitri.shakhmatov@ehime-u.ac.jp}
\thanks{The first listed author was partially supported by the Grant-in-Aid for Scientific Research~(C) No.~20K03615 by the Japan Society for the Promotion of Science (JSPS)}

\author[V.H. Ya\~nez]{V\'{\i}ctor Hugo Ya\~nez}
\address{Institute of Mathematics \\
Nanjing Normal University, Nanjing 210046, China}
\email{victor\textunderscore yanez@comunidad.unam.mx}



\begin{abstract}
We prove that every unconditionally closed subset of a free
group is algebraic, thereby answering affirmatively a 77 years old problem of Markov for free groups. In modern terminology, this means that Markov and Zariski topologies coincide in free groups.
It follows
that the class of groups for which Markov and Zariski topologies coincide 
is not closed under taking quotients.  
We also show that 
Markov and Zariski 
topologies differ from the so-called precompact Markov topology in non-commutative free groups.
\end{abstract}

\maketitle

\section{Introduction and main results}

Let $G$ be a group. Pick $x\notin G$ and consider the free product $G\wprod \langle x\rangle$ of $G$ with the cyclic group $\langle x\rangle \cong \Z$ generated by $x$. For every $g\in G$, let $\mathrm{ev}_g: 
G\wprod \langle x\rangle\to G$ be the unique homomorphism which sends $x$ to $g$ and coincides with the identity on $G$. 
For every word $w\in G\wprod \langle x\rangle$, the set
$E_w=\{g\in G: \mathrm{ev}_g(w)=1\}$ is 
the solution set in $G$ of the equation $w=1$ with a single variable $x$ and coefficients taken from $G$.
(Here $1$ denotes the identity of $G$.)
The coarsest topology $\mathfrak{Z}_G$ on $G$ such that all solution sets $E_w$ for $w\in G\wprod \langle x\rangle$ are $\mathfrak{Z}_G$-closed is called the {\em verbal\/} \cite{Bryant}
or {\em Zariski\/} \cite{BMR, DS-OPIT2, DS_R, DS_MZ} topology of $G$. 
In \cite{Markov44, Markov45, Markov46},
Markov calls 
the solution set $E_w$ for a word $w\in G\wprod \langle x\rangle$ an {\em elementary algebraic set\/} in $G$
and he calls
$\mathfrak{Z}_G$-closed sets {\em algebraic\/} subsets of $G$.

Dikranjan and Toller in 
\cite[Section 5.3]{DT-Ischia} 
derived the following explicit description of the Zariski topology of free groups from 
results
of Bryant \cite{Bryant}, Guba  \cite{Guba}, and 
Chiswell and Remeslennikov \cite{CR}:
\begin{theorem}
\label{derived:from:Guba}
Zariski closed subsets of a free group $G$ are finite unions of singletons or 
sets of the form 
$aC_G(b)c$, 
where $a,b,c\in G$ and
$C_G(b)
=\{g\in G: bg=gb\}
$ is the centralizer of $b$ in $G$.
\end{theorem}

It was noted in \cite[Section 5.4.1]{Toller} that, without loss of generality, one can take $c=1$ in Theorem \ref{derived:from:Guba}.
We also note that the set $C_G(b)$ in Theorem \ref{derived:from:Guba} is an infinite cyclic subgroup of $G$ containing $b$ but it need not coincide with the cyclic subgroup of $G$ generated by $b$.

In \cite{Markov44, Markov45, Markov46},
Markov 
says that
a subset $S$ of $G$ is {\em unconditionally closed\/} in $G$ if $S$ is closed in {\em every\/} Hausdorff group topology on $G$.
The family of all unconditionally closed subsets of $G$ coincides with the family of $\mathfrak{M}_G$-closed sets of a unique topology $\mathfrak{M}_G$ on $G$ called its {\em Markov} topology
in \cite{DS-OPIT2, DS_R, DS_MZ}. 

It is clear that 
algebraic subsets of $G$ are unconditionally closed, so the inclusion $\mathfrak{Z}_G\subseteq \mathfrak{M}_G$ holds. 
In 1945 
Markov  \cite{Markov45} asked whether 
every unconditionally closed 
set
is algebraic, or equivalently, whether the equality $\mathfrak{Z}_G= \mathfrak{M}_G$ holds for every group $G$. In 
 \cite{Markov46}
he gave a positive answer for all countable groups, and 
in \cite{Markov45} he
attributed to Perel'man a positive answer for all abelian groups. However, Perel'man's result remained unpublished, so 
proofs were given only much later in 
\cite{Sipacheva, DS_R}. 

Clearly, the Markov topology $\mathfrak{M}_G$ of a group $G$ is discrete if and only if $G$ is {\em non-topologizable\/}, i.e., does not admit a non-discrete Hausdorff group topology.  
Whether non-topologizable groups exist was another old problem of Markov.
Shelah \cite{Shelah} gave a
first
 example of 
a 
non-topologizable group 
(of size continuum)
under the Continuum Hypothesis CH.
Later 
examples of non-topologizable groups were constructed in ZFC
by Ol'shanskii \cite{Ol}
and Hesse \cite{Hesse}.
Ol'shanskii's example is countable.
\begin{example}[Hesse
\cite{Hesse}]
\label{exa:Hesse}
For every regular cardinal $\kappa$, 
there exists a
non-topologizable group $G_\kappa$
such that the singleton $\{1\}$ cannot be represented as an intersection of
$\le\kappa$-many $\mathfrak{Z}_{G_\kappa}$-open sets.
In particular, 
$\mathfrak{Z}_{G_\kappa}$ is non-discrete, while
$\mathfrak{M}_{G_\kappa}$ is discrete, so
$\mathfrak{Z}_{G_\kappa}\neq \mathfrak{M}_{G_\kappa}$. 
\end{example}
Among all non-topologizable groups mentioned above, only Hesse's example resolves negatively the Markov's 
$\mathfrak{Z}_G= \mathfrak{M}_G$ problem.
Indeed,
$\mathfrak{Z}_G=\mathfrak{M}_G$ for the Ol'shanskii group $G$, as $G$ is countable, 
and the equality $\mathfrak{Z}_G=\mathfrak{M}_G$ holds for all countable groups \cite{Markov45}.
However,
Sipacheva \cite{Sip2007} showed that 
a suitable modification of
Shelah's construction 
yields an
example of a non-topologizable group $G$ 
having
a non-discrete 
Zariski topology $\mathfrak{Z}_G$,
thereby providing 
a 
counter-example to Markov's problem
under CH.
This may provide further context
due to the fact that
Hesse's
ZFC 
counter-example to Markov's problem remained unpublished (except for PhD thesis \cite{Hesse}; 
our formulation of Example \ref{exa:Hesse} is extracted from
\cite[Theorem 3.1]{DT-Ischia}).

Motivated by 
Example \ref{exa:Hesse},
Dikranjan and the first listed author asked in \cite{DS-OPIT2} whether the answer to Markov's problem is affirmative at least for permutation groups and free groups.
Banakh, Guran and Protasov \cite{BGP} proved that Markov and Zariski topologies coincide for all permutation groups, thereby answering 
\cite[Question 38]{DS-OPIT2}.

In this paper, 
we answer 
\cite[Question 39]{DS-OPIT2}, repeated also as 
\cite[Question 8.1]{DT-Ischia},
by confirming Markov's conjecture for all free (non-commutative) groups.

\begin{theorem}
\label{main:theorem}
Every unconditionally closed subset of a free group is algebraic.
Therefore, Markov and Zariski topologies coincide in free groups.
\end{theorem}

When combined with Theorem \ref{derived:from:Guba}, this result provides the complete description of the Markov topology of free groups.

The proof of 
Theorem \ref{main:theorem}
will be given in 
Section \ref{section:proof}.

Recall that a topological group is \emph{precompact} if it is a subgroup of some compact group.
Due to the importance of precompact groups in the theory of topological groups, Dikranjan and the first author introduced
in \cite{DS-OPIT2, DS_MZ} the ``precompact analogue'' of the Markov topology of a group $G$ called
its \emph{precompact Markov topology} $\mathfrak{P}_G$.
The $\mathfrak{P}_G$-closed subsets of $G$ are precisely those subsets of $G$ 
which are closed in every \emph{precompact} 
Hausdorff
group topology on $G$.
Clearly, $\mathfrak{P}_G$ is discrete if $G$ is a finite group. 
For an infinite group $G$, $\mathfrak{P}_G$ is non-discrete if and only if homomorphisms of $G$ into compact groups separate points of $G$, or equivalently, if $G$ is maximally almost periodic in its discrete topology (Corollary \ref{non-discretness:of:P_G}).
Obviously,
the inclusion $\mathfrak{M}_G\subseteq \mathfrak{P}_G$ 
always holds. The equality $\mathfrak{M}_G= \mathfrak{P}_G$ holds for every abelian group $G$ \cite{DS_MZ}
but fails already for permutation groups \cite{DT-2012}.
Our second main result shows that the equality $\mathfrak{M}_G= \mathfrak{P}_G$ fails also for free groups $G$ of rank at least $2$.

\begin{theorem}
\label{precompact:Markov}
$\mathfrak{M}_G\neq \mathfrak{P}_G$ for every free group $G$ of rank at least $2$.
\end{theorem}

The restriction on rank of the free group in Theorem \ref{precompact:Markov} is essential, as
$\mathfrak{M}_{\mathbb{Z}}= \mathfrak{P}_{\mathbb{Z}}$
for 
the free group $\mathbb{Z}$ of rank $1$; see \cite{DS_MZ}.
We note that the precompact Markov topology
$\mathfrak{P}_G$ is non-discrete for a non-trivial free group $G$
(Corollary \ref{non-discrete:precompact:topology:on:free:groups}(b)). 
To the best of the 
authors'
knowledge,
all examples of groups $G$ satisfying $\mathfrak{M}_G\neq \mathfrak{P}_G$  known so far had discrete precompact Markov topology $\mathfrak{P}_G$; see Remarks \ref{rem:11.1:h}(i) and \ref{rem:6.3}(ii).

The proof of 
Theorem \ref{precompact:Markov}
will be given in 
Section \ref{P_G:section}.
It relies on 
the existence of a non-trivial Bohr convergent sequence in the free group with 2 generators
due to Thom \cite{Thom}.

It seems worthwhile to point out that
our proofs of 
Theorems \ref{main:theorem}
and 
\ref{precompact:Markov}
do not depend on the description of the Zariski topology of free groups obtained in Theorem 
\ref{derived:from:Guba}.

Additional results on
$\mathfrak{Z}_G$,
$\mathfrak{M}_G$
and 
$\mathfrak{P}_G$
for a non-commutative group $G$ can be found in 
\cite{DT-2012, DT, {BDT}}.

The paper is organized as follows. 
In Section \ref{sec:MZ}, we deduce from Theorem \ref{main:theorem} that \emph{every} group is a quotient
of a group with coinciding Markov and Zariski topologies (Corollary \ref{universal:quotient}), thereby establishing that the class of groups with coinciding Markov and Zariski topologies is not closed
under taking quotient groups (Corollary \ref{corollary:1.4}).
A similar result holds for subgroups as well (Theorem \ref{every:group:is:a:quotient:of:MZ:group}).
The main result 
of
Section \ref{extension:section} 
is Theorem \ref{second:lemma}
which is obtained by combining 
Lemma \ref {first:lemma} on monomorphisms of a free group into a power of another free group
with the result of Slutsky \cite{Slutsky} on extension of Polish group topologies on two groups to a Polish group topology of their free product.
In turn, 
Theorem \ref{second:lemma}
is combined with Fact \ref{Hausdroff:embedding:theorem}
proved in 
\cite[Theorem 3.9]{DS_R} in order to establish Corollary \ref{main:thm:2}.
The latter theorem serves as a foundation for applying a corollary of a certain Reflection Principle from \cite{DS_R} 
stating
that Markov and Zariski topologies of a group $G$ 
coincide provided that the family of all at most countable subgroups of $G$ contains a closed unbounded subset 
consisting of Hausdorff embedded subgroups of $G$ (Corollary
\ref{Hausdorff:embedded:corollary}).
In Lemma \ref{natural:club}
we define
a natural closed unbounded subset $\mathcal{C}$ of the family 
of all at most countable subgroups of a free group $G$,
and
we apply Corollary \ref{main:thm:2}
to prove that all members of $\mathcal{C}$ are Hausdorff embedded in $G$, thereby deducing Theorem \ref{main:theorem} from Corollary
\ref{Hausdorff:embedded:corollary}.
In Section \ref{P_G:section} we establish connections between 
the precompact Markov topology $\mathfrak{P}_G$ of a group $G$ and its Bohr topology $\mathfrak{B}_G$.
The key result here is Theorem \ref{homeo:corollary}.
Its
Corollary \ref{the:same:topologies:on:Bohr:compact:subsets} states that if
a group $G$ is maximally almost periodic group in its discrete topology, then $\mathfrak{B}_G$ and $\mathfrak{P}_G$ induce the same (subspace) topology on every Bohr compact subset 
of $G$. 
Another Corollary \ref{sequences:in:P_G} combines Theorem \ref{homeo:corollary} with the result
of Thom \cite{Thom} on the existence of a non-trivial sequence in the free group with two generators converging to its identity element in the Bohr topology to conclude that 
a non-commutative free group $G$ has an infinite Hausdorff subspace in its precompact Markov topology $\mathfrak{P}_G$.
On the other hand, the Zariski topology $\mathfrak{Z}_G$ of a free group $G$ is known to be Noetherian \cite[Corollary 5.2]{DT-Ischia}, and Noetherian spaces have no infinite Hausdorff subspaces.
It follows that $\mathfrak{Z}_G\neq \mathfrak{P}_G$ for a non-commutative free group $G$. Combining this with our Theorem \ref{main:theorem}, we deduce Theorem \ref{precompact:Markov}.
In Section \ref{sec:11}, we introduce two new classes of groups related to the precompact Markov topology and formulate
a series of questions 
related to them.

Thereafter, the symbol $F(X)$ denotes the free 
group over a set $X$. 

\section{The class $\mathcal{MZ}$ of groups with coinciding Markov and Zariski topologies}
\label{sec:MZ}

Following \cite{{DS_R}, DT-Ischia}, we denote by $\mathcal{MZ}$ the class of groups for which Markov and Zariski topologies coincide.
As a corollary of Theorem 
\ref{main:theorem}, 
we can
show
 that the quotients of groups from this class generate the class of all groups.

\medskip
\begin{corollary}
\label{universal:quotient}
Every group is a quotient group of some group 
from
the class $\mathcal{MZ}$.
\end{corollary}
\begin{proof}
Let $H$ be an arbitrary group, and let $G$ be the free group $F(H)$ over the alphabet set $H$. Then the identity map $i:H\to H$ can be extended to a homomorphism $f: G\to H$ between $G=F(H)$ and $H$. Since $i$ is surjective, so is $f$. Therefore,
$H$ is a quotient group of $G$. Since $G$ is a free group, it belongs to the class $\mathcal{MZ}$  by Theorem \ref{main:theorem}.
\end{proof}

\begin{corollary}
\label{corollary:1.4}
A quotient 
group of a group 
from
the class $\mathcal{MZ}$ need not belong to the class $\mathcal{MZ}$.
\end{corollary}
\begin{proof}
Let $H$ be  
the group 
from Example \ref{exa:Hesse}.
Then $H\notin \mathcal{MZ}$. On the other hand, $H$ is a quotient group of some group $G\in \mathcal{MZ}$ by Corollary \ref{universal:quotient}.
\end{proof}

A result similar to Corollary \ref{universal:quotient} holds for subgroups as well.
\begin{theorem}
\label{every:group:is:a:quotient:of:MZ:group}
Every group  is a subgroup of some group in the class $\mathcal{MZ}$.
\end{theorem}
\begin{proof}
The proof is essentially the same as that of \cite[Corollary 3.2]{BGP}. Indeed,
every group $G$ is isomorphic to a 
subgroup of the symmetric group $S(G)$ over the set $G$
by Cayley's theorem,
and
$S(G)$ belongs to the class $\mathcal{MZ}$
by 
\cite[Corollary 3.1]{BGP}.
\end{proof}

Since Hesse's 
Example \ref{exa:Hesse} 
does not belong to the class $\mathcal{MZ}$, one derives the analog of Corollary \ref{corollary:1.4} for subgroups obtained by Banakh, Guran and Protasov:
\begin{corollary}
\label{subgroup:corollary}
{\rm (\cite[Corollary 3.2]{BGP})}
A subgroup of a group from the class $\mathcal{MZ}$ need not belong to the class $\mathcal{MZ}$.
\end{corollary}

\section{Extension of separable metric group topologies in free groups}
\label{extension:section}

\begin{fact}{\cite[Theorem 3.8]{Slutsky}}  \label{slutsky}
Let $G$ and $H$ be 
Polish
groups. Then there exist a Polish group $T$ and topological group embeddings $\psi_G: G \to T$, $\psi_H: H \to T$ such that the group 
generated by $\psi_G(G)\cup\psi_H(H)$
is naturally isomorphic to the free product 
$G \wprod H$ of $G$ and $H$. 
\end{fact}

\begin{corollary}
\label{corollary:from:Slytsky}
If $A$ and $B$ are separable metric groups, then there exists a separable metric group topology on their free product $A\wprod B$ which induces 
the original topologies on $A$ and $B$.
\end{corollary}
\begin{proof}
Let $G$ and $H$ denote the completions of $A$ and $B$, respectively. Then $G$ and $H$ are Polish groups.
Applying Fact \ref{slutsky} to $G$ and $H$, we can find a Polish group $T$ and topological embeddings $\psi_G: G \to T$, $\psi_H: H \to T$
as in the conclusion of this theorem. Since 
$A\subseteq G$, $B\subseteq H$ and
the group 
generated by $\psi_G(G)\cup\psi_H(H)$
is naturally isomorphic to the free product 
$G \wprod H$ of $G$ and $H$, it follows that the 
group 
$N$
generated by $\psi_G(A)\cup\psi_H(B)$
is naturally isomorphic to the free product 
$A \wprod B$ of $A$ and $B$.
Now the subspace topology induced by $T$ on 
$N$
is a separable metric group topology
inducing the original topologies 
on $A$ and $B$.
\end{proof}

Recall that a {\em monomorphism} is a homomorphism with 
trivial kernel, or equivalently, an injective homomorphism.

\begin{lemma}
\label{first:lemma}
Let 
$X$ be a non-empty set and 
$Y$ be an infinite set.
Then 
there exist a non-empty set 
$G$
and a  
monomorphism
$\varphi: F(X)\to F(Y)^{G}$ 
such that 
$\pi_g\circ \varphi (x)=x$ 
whenever $x\in X\cap Y$
and $g\in G$, 
where 
$\pi_g:
F(Y)^G\to F(Y)$ is the projection on the coordinate $g$. 
\end{lemma}
\begin{proof}
The following claim is the key to our proof.
\begin{claim}
\label{new:claim}
For every $g\in F(X)\setminus\{e\}$ there exists a homomorphism
$\varphi_g:F(X)\to F(Y)$ satisfying the following two conditions:
\begin{itemize}
\item[(i)] $\varphi_g(g)\neq e$,
\item[(ii)] $\varphi_g(x)=x$ for every $x\in X\cap Y$.
\end{itemize}
\end{claim}
\begin{proof}
Fix $g\in F(X)\setminus\{e\}$.
Since $g\neq e$, there exist a positive integer $n$, elements
$x_1,\dots,x_n\in X$ and numbers $\varepsilon_1,\dots,\varepsilon_n\in\{-1,1\}$ such that 
\begin{equation}
\label{representation:of:g}
g=x_1^{\varepsilon_1} x_2^{\varepsilon_2}\dots x_n^{\varepsilon_n}
\end{equation}
is an irreducible representation of $g$; that is,
$\varepsilon_j\neq -\varepsilon_{j+1}$ whenever $j=1,\dots,n-1$
and $x_j=x_{j+1}$.
Let 
\begin{equation}
\label{def:of:E}
E=\{x_1,x_2,\dots,x_n\}.
\end{equation}
Since $E$ is a finite subset of $X$ and $Y$ is infinite,
one can fix a map $f:X\to Y$ such that
\begin{itemize}
\item[(a)]
$f \restriction_E$
is an injection, 
and 
\item[(b)]
$f(x)=x$ for each $x\in X\cap Y$.
\end{itemize}
Let $\varphi_g: F(X)\to F(Y)$ be the homomorphism extending 
$f$,
i.e. satisfying
\begin{itemize}
\item[(c)]
$\varphi_g\restriction_X=f$.
\end{itemize}
Let $y_j=f(x_j)$ for $j=1,\dots,n$.
Then $y_1,y_2,\dots,y_n\in Y$.
Since $\varphi_g$ is a homomorphism, it follows from (c) that
\begin{equation}
\label{eq:varphi:of:g}
\varphi_g(g)=y_1^{\varepsilon_1} y_2^{\varepsilon_2}\dots y_n^{\varepsilon_n}.
\end{equation}
Since \eqref{representation:of:g} is an irreducible
representation of $g$, 
it follows from \eqref{def:of:E}, (a) and the definition of $y_j$
for  $j=1,\dots,n$ that \eqref{eq:varphi:of:g}
is an irreducible
representation of $\varphi_g(g)$. In particular, $\varphi_g(g)\neq e$. This establishes (i).

(ii) follows from (b) and (c).
\end{proof}

Let $G=F(X)\setminus\{e\}$. Since $X\neq\emptyset$, 
$F(X)$ is non-trivial, and so $G\neq\emptyset$.
For each $g\in G$, choose a homomorphism $\varphi_g$ as in the conclusion of Claim \ref{new:claim}.
Let $\varphi: F(X)\to F(Y)^G$ be the unique homomorphism
such that
\begin{equation}
\label{eq:6:b}
\pi_g\circ \varphi=\varphi_g
\ \ 
\text{ for all }
g\in G.
\end{equation}

\begin{claim}
$\varphi$ is a monomorphism.
\end{claim}
\begin{proof}
Suppose that $g\in F(X)$ and $g\neq e$. Then $g\in G$
and 
\begin{equation}
\label{eq:5:d}
\pi_g(\varphi(g))=\pi_g\circ \varphi(g)=\varphi_g(g)\neq e
\end{equation}
by \eqref{eq:6:b} and item (i) of Claim \ref{new:claim}.
Since $\pi_g$ is a homomorphism, \eqref{eq:5:d} implies $\varphi(g)\neq e$.
We have proved that the kernel of $\varphi$ is trivial.
\end{proof}

Suppose that $x\in X\cap Y$
and $g\in G$.
Then 
$\pi_g\circ \varphi (x)=\varphi_g(x)=x$
by \eqref{eq:6:b} and 
 item (ii) of Claim \ref{new:claim}.
\end{proof}

Recall that for every subset $Z$ of a set $X$, the subgroup of $F(X)$ generated by $Z$ coincides with $F(Z)$.
In what follows, we shall always identify $F(Z)$ with this subgroup of $F(X)$.

\begin{theorem}
\label{second:lemma}
Let $X$ be a set and $Z$ be its subset. Then for every separable metric group topology $\mathscr{T}$ on $F(Z)$ 
there exists 
a Hausdorff group topology $\mathscr{T}'$ on $F(X)$ inducing the original topology $\mathscr{T}$ on $F(Z)$;
that is, $(F(Z),\mathscr{T})$ is a subspace of 
$(F(X),\mathscr{T}')$.
\end{theorem}
\begin{proof}
Without loss of generality, we may assume that $X$ is non-empty.
Fix a set $Y$ such that $Y\cap X=Z$ and $Y\setminus X$ is countably infinite. Note that $F(Y)=F(Z)\wprod F(Y\setminus X)$.
Let $A$ be the group $F(Z)$ equipped with the topology $\mathscr{T}$, 
and let $B$ be the group 
$F(Y\setminus X)$ equipped with the discrete topology.
Since $Y\setminus X$ is countably infinite, $F(Y\setminus X)$ is a countable group, so $B$ is a separable metric group.
Use Corollary \ref{corollary:from:Slytsky}
to fix a separable metric group topology $\mathscr{T}^*$ on $F(Y)=A\wprod B$ which induces the original topologies on $A$ and $B$. In particular, $(F(Z),\mathscr{T})$ is a subgroup of $(F(Y),\mathscr{T}^*)$. Let 
$G$
 and $\varphi$ be 
as in the conclusion of Lemma \ref{first:lemma}.
Let $\mathscr{T}_\Pi$ be the Tychonoff product topology on $F(Y)^G$; that is,
$(F(Y)^G,\mathscr{T}_\Pi)=(F(Y),\mathscr{T}^*)^G$. 
Since $\varphi$ is a monomorphism,
$\mathscr{T}'=\{\varphi^{-1}(U):U\in \mathscr{T}_\Pi\}$ is a Hausdorff group topology on $F(X)$.

Let 
$g\in G$
be arbitrary. 
By the conclusion of Lemma \ref{first:lemma},
the composition $\pi_g\circ \varphi$ is the identity map on $Z$.
Since both $\varphi$ and $\pi_g$ are group homomorphisms,
this implies that $\pi_g\circ \varphi$ is the identity map on $F(Z)$ as well. Since $(F(Z),\mathscr{T})$ is a subgroup of $(F(Y),\mathscr{T}^*)$, we conclude that 
$\pi_g\circ \varphi\restriction_{F(Z)}: (F(Z),\mathscr{T})\to (F(Y),\mathscr{T}^*)$ is a homeomorphic embedding.
Since this holds for an arbitrary $g \in G$,
it follows that the map
$\varphi\restriction_{F(Z)}:(F(Z),\mathscr{T})\to (F(Y)^G,\mathscr{T}_\Pi)$ is a homeomorphic embedding.
It follows from this that $(F(Z),\mathscr{T})$ is a subspace of 
$(F(X),\mathscr{T}')$.
\end{proof}

\section{Hausdorff and Markov embeddings}

For a subset $Y$ of a topological space $(X,\mathscr{T})$ the symbol $\mathscr{T}\restriction_Y$ denotes the subspace topology 
which $Y$ inherits from $(X,\mathscr{T})$; that is,
$\mathscr{T}\restriction_Y=\{U\cap Y:U\in\mathscr{T}\}$.

\begin{definition}\cite[Definition 3.1(i)]{DS_R}
\label{def:hausdorff:embeddings}
A subgroup $H$ of a group $G$ is called \emph{Hausdorff embedded in $G$} if 
for
every Hausdorff group topology $\mathscr{T}$ on $H$ 
there exists a
Hausdorff group topology $\mathscr{T}'$ on $G$ 
such that $\mathscr{T}=\mathscr{T}'\restriction_H$
(and in this case we say that $\mathscr{T'}$ \emph{extends} $\mathscr{T}$).
\end{definition}

\begin{fact}\cite[Theorem 3.9]{DS_R}
\label{Hausdroff:embedding:theorem}
Let $H$ be an at most countable subgroup of a group $G$. If every metric group topology on $H$ can be extended to a (not necessarily metric) group topology on $G$, then $H$ is Hausdorff embedded in $G$.
\end{fact}

\begin{corollary} \label{main:thm:2}
Let $X$ be a set and $Z$ be its at most countable subset. Then every Hausdorff group topology on $F(Z)$ can be extended to a Hausdorff group topology on $F(X)$; that is, $F(Z)$ is Hausdorff embedded in $F(X)$.
\end{corollary}
\begin{proof}
The result follows
from Theorem \ref{second:lemma},
Definition \ref{def:hausdorff:embeddings}
and Fact \ref{Hausdroff:embedding:theorem}.
\end{proof}

\begin{remark}
Dekui Peng suggested that free products of topological groups can be used to extend Corollary \ref{main:thm:2} to arbitrary subsets $Z$ of the set $X$. Indeed, this turned out to be the case.
Let $Z$ be an arbitrary subset of a set $X$ and let 
$\mathscr{T}$
 be a Hausdorff group topology on $F(Z)$. 
By \cite[Theorem 1]{Graev}, there exists a Hausdorff group topology 
$\mathscr{T}'$ 
on the
free product 
$F(Z) \wprod F(X \setminus Z)$
of $F(Z)$ and $F(X \setminus Z)$
which induces the original topology 
$\mathscr{T}$
on $F(Z)$ and the discrete topology on $F(X\setminus Z)$.
Since $F(X) = F(Z) \wprod F(X \setminus Z)$,
the Hausdorff group topology $\mathscr{T}'$ on $F(X)$ extends 
the topology $\mathscr{T}$ on $F(Z)$. 
Therefore,
$F(Z)$ is a Hausdorff embedded in $F(X)$
by Definition \ref{def:hausdorff:embeddings}. 
This finishes the proof of the general version of Corollary \ref{main:thm:2}.

Even though this proof is quite short and avoids the need for results from Section \ref{extension:section}, it relies on the theorem of Graev from \cite{Graev} whose elaborate proof is available only in Russian.
Since we only require the limited version of Corollary \ref{main:thm:2}
with at most countable set $Z$ in this paper, we decided to keep our original proof which uses only sources readily available in English.
\end{remark}

\begin{definition}
\cite[Definition 2.1(ii)]{DS_R}
\label{def:Markov:embeddings}
A subgroup $H$ of a group $G$ is called \emph{Markov embedded in $G$} if $(H,\mathfrak{M}_H)$ is a subspace of $(G,\mathfrak{M}_G)$.
\end{definition}

Item (iii) of our next lemma is a limited version of \cite[Lemma 3.6]{DS_R}. Since \cite[Lemma 3.6]{DS_R} is stated without proof,
we include the proof 
below
for the convenience of the reader.

\begin{lemma}
\label{this:lemma:sep13}
Let $H$ be a subgroup of a group $G$.
\begin{itemize}
\item[(i)]
If $\mathfrak{Z}_H=\mathfrak{M}_H$, then $\mathfrak{M}_H\subseteq \mathfrak{M}_G\restriction_H$.
\item[(ii)]
If $H$ is Hausdorff embedded in $G$, then 
$\mathfrak{M}_H\supseteq \mathfrak{M}_G\restriction_H$.
\item[(iii)] If $H$ is at most countable and Hausdorff embedded in $G$, then
$H$ is also Markov embedded in $G$.
\end{itemize}
\end{lemma}
\begin{proof}
(i) 
$\mathfrak{M}_H=\mathfrak{Z}_H$ by the assumption of item (i). Note that $\mathfrak{Z}_H\subseteq \mathfrak{Z}_G\restriction_H$. Since $\mathfrak{Z}_G\subseteq \mathfrak{M}_G$, we have $\mathfrak{Z}_G\restriction_H\subseteq \mathfrak{M}_G\restriction_H$. This gives the desired inclusion.

\medskip
(ii) 
Let $X$ be a $\mathfrak{M}_G\restriction_H$-closed subset of $H$. Then there exists a $\mathfrak{M}_G$-closed subset $Y$ of $G$ such that $Y\cap H=X$.
Let $\mathscr{T}$ be a Hausdorff group topology on $H$.
Since $H$ is Hausdorff embedded in $G$, we can use 
Definition \ref{def:hausdorff:embeddings} to find a Hausdorff group topology $\mathscr{T}'$ on $G$ extending $\mathscr{T}$.
Since $Y$ is $\mathfrak{M}_G$-closed, it is also 
$\mathscr{T}'$-closed in $G$.
Therefore, $X=Y\cap H$ is $\mathscr{T}$-closed.
We have checked that $X$ is $\mathscr{T}$-closed in every Hausdorff group topology $\mathscr{T}$ on $H$,
Thus,
$X$ is $\mathfrak{M}_H$-closed.

\medskip
(iii) Suppose that $H$ is at most countable and Hausdorff embedded in $G$. From the first assumption we get $\mathfrak{Z}_H=\mathfrak{M}_H$, as Markov and Zariski topologies coincide for at most countable groups by an old theorem of Markov \cite{Markov46}.
Combining this with (i), we obtain the inclusion $\mathfrak{M}_H\subseteq \mathfrak{M}_G\restriction_H$.
Since $H$ is Hausdorff embedded in $G$, the inverse inclusion 
$\mathfrak{M}_H\supseteq \mathfrak{M}_G\restriction_H$ holds by (ii). This establishes the equality 
$\mathfrak{M}_H= \mathfrak{M}_G\restriction_H$.
Therefore, $(H,\mathfrak{M}_H)$ is a subspace of $(G,\mathfrak{M}_G)$, and so $H$ is Markov embedded in $G$ by 
Definition \ref{def:Markov:embeddings}.
\end{proof}

\begin{remark}
Recall that a subgroup $H$ of a group $G$ is {\em Zariski embedded\/} in $G$ if the Zariski topology of $G$ induces on $H$ the Zariski topology of $H$ \cite[Definition 
2.1(i)]{DS_R}. %
The authors have proved recently that \emph{every} subgroup of a free group is both Zariski embedded and Markov embedded in it \cite{SY}. However, the free group with two generators contains a normal subgroup which is not Hausdorff embedded in it
\cite{SY}. This shows the limits of our Corollary \ref{main:thm:2}.
\end{remark}

\section{Proof of Theorem \ref{main:theorem}}
\label{section:proof}

The following definition recalls well-known set-theoretic concepts.

\begin{definition}
\label{def:club}
Let $G$ be a set and $\mathcal{C} \subseteq [G]^{\leq \omega}$,
where $[G]^{\leq \omega}$ denotes the set of all at most countable subsets of $G$.
\begin{itemize}
\item[(i)] $\mathcal{C}$ is \emph{closed in} $[G]^{\leq \omega}$ if, whenever 
$\{C_n: n \in \N\} \subseteq \mathcal{C}$ \text{ and } 
$C_0 \subseteq C_1 \subseteq \dots \subseteq C_n \subseteq C_{n+1} \subseteq \dots$,
then $\bigcup \{C_n: n \in \N \} \in \mathcal{C}$,
\item[(ii)] $\mathcal{C}$ is \emph{unbounded in $[G]^{\leq \omega}$} if for every $H \in [G]^{\leq \omega}$ there exists $C \in \mathcal{C}$ with $H \subseteq C$, 
\item[(iii)] $\mathcal{C}$ is a \emph{club in $[G]^{\leq \omega}$} (a common abbreviation for ``closed and unbounded'') if $\mathcal{C}$ is both closed and unbounded in $[G]^{\leq \omega}$.
\end{itemize}
\end{definition}

The following lemma gives an example of a club in $F(X)^{\le \omega}$.

\begin{lemma}
\label{natural:club}
Let $X$ be a set and let $G=F(X)$.
Then
the family 
\begin{equation}
\label{eq:family:C}
\mathcal{C}=\{F(Z):Z\in[X]^{\le\omega}\}
\end{equation}
 is a club in $[
G
]^{\le \omega}$. 
\end{lemma}
\begin{proof}

Suppose that
\begin{equation}
\label{chain:of:Cs}
\{C_n: n \in \N\} \subseteq \mathcal{C}
\end{equation}
 and  
\begin{equation}
\label{increasing:chain}
C_0 \subseteq C_1 \subseteq \dots \subseteq C_n \subseteq C_{n+1} \subseteq \dots. 
\end{equation}
By \eqref{eq:family:C} and \eqref{chain:of:Cs}, 
for every $n\in\N$ there exists $Z_n\in [X]^{\le \omega}$ such that 
\begin{equation}
\label{C_n:Z_n}
C_n=F(Z_n).
\end{equation}
 From  \eqref{increasing:chain}  and \eqref{C_n:Z_n} it follows that
\begin{equation}
\label{increasing:chain:Z}
Z_0 \subseteq Z_1 \subseteq \dots \subseteq Z_n \subseteq Z_{n+1} \subseteq \dots.
\end{equation}
Since each $Z_n$ is at most countable subset of $X$, so is the set
$Z=\bigcup_{n\in\N} Z_n$.
Combining this with \eqref{eq:family:C}, we conclude that $F(Z)\in\mathcal{C}$. It easily follows from \eqref{increasing:chain:Z} and 
our definition of $Z$ that
$F(Z)=\bigcup_{n\in\N} F(Z_n)$.
From this and \eqref{C_n:Z_n}, we get
$\bigcup \{C_n: n \in \N \} = \bigcup_{n\in\N} F(Z_n)=F(Z)\in \mathcal{C}$.
By Definition \ref{def:club}(i), this shows that $\mathcal{C}$ is closed in $[G]^{\leq \omega}$.

Let 
$H \in [G]^{\leq \omega}$. For every $h\in H$, denote by $S_h$ the finite set of letters in the alphabet $X$ appearing in the irreducible representation of $h$. Then $Y=\bigcup_{h\in H} S_h\in[X]^{\le\omega}$, so $F(Y)\in\mathcal{C}$ by \eqref{eq:family:C}.
Since $h\in F(S_h)\subseteq F(Y)$ for every $h\in H$, the inclusion $H\subseteq F(Y)$ holds.
By Definition \ref{def:club}(ii), this shows that $\mathcal{C}$ is unbounded in $[G]^{\leq \omega}$.

We established that $\mathcal{C}$ is both closed and unbounded in $[G]^{\leq \omega}$. By Definition \ref{def:club}(iii), this means that 
$\mathcal{C}$ is a club in $[G]^{\leq \omega}$.
\end{proof}

The following fact is a corollary of a reflection 
principle 
established in \cite{DS_R}.

\begin{fact}
\cite[Corollary 5.3]{DS_R}
 \label{this:coro:sep13}
Let $G$ be a group such that the family 
\begin{equation} \label{this:eq2:sep13}
\mathscr{M}(G) = \{H \in 
[G]^{\le\omega}:
H \text{ is a Markov embedded subgroup of } G\}
\end{equation}
contains 
a
club in $[G]^{\leq \omega}$. Then $\mathfrak{Z}_{G}=\mathfrak{M}_{G}$; that is, 
Markov and Zariski topologies on $G$ coincide.
\end{fact}

The following corollary was implicitly used in some arguments in
\cite{DS_R} even though it was not stated in this paper explicitly.

\begin{corollary}
\label{Hausdorff:embedded:corollary}
Let $G$ be a group such that the family 
\begin{equation} \label{eq:Hausdorff}
\mathscr{H}(G) = \{H \in 
[G]^{\le\omega}:
H \text{ is a Hausdorff embedded subgroup of } G\}
\end{equation}
contains a club in $[G]^{\leq \omega}$. Then $\mathfrak{Z}_{G}=\mathfrak{M}_{G}$; that is, 
Markov and Zariski topologies on $G$ coincide.
\end{corollary}
\begin{proof}
Note that $\mathscr{H}(G)\subseteq \mathscr{M}(G)$ by Lemma
\ref{this:lemma:sep13}(iii), so we can apply Fact \ref{this:coro:sep13}.
\end{proof}

\noindent
{\bf Proof of Theorem \ref{main:theorem}:}
Let $X$ be 
a
set
and let $G=F(X)$.
Let
$\mathcal{C}$ be the club in $[
G
]^{\le \omega}$ defined in Lemma \ref{natural:club}.

Let $Z\in [X]^{\le\omega}$ be arbitrary. Since $Z$ is an at most countable subset of $X$, $F(Z)$ is a countable subgroup of
$G$. Applying Corollary \ref{main:thm:2},
we conclude that 
$F(Z)$ is 
Hausdorff
embedded in $G$.
Therefore, 
$F(Z)\in\mathscr{H}(G)$ 
by \eqref{eq:Hausdorff}.
Recalling \eqref{eq:family:C}, we conclude that
$\mathcal{C}\subseteq \mathscr{H}(G)$.

We have proved that the family 
$\mathscr{H}(G)$ 
contains a club
$\mathcal{C}$ in $[G]^{\le\omega}$.
Therefore,
$\mathfrak{Z}_{G}=\mathfrak{M}_{G}$ 
by 
Corollary \ref{Hausdorff:embedded:corollary}.
\qed

\section{The
interplay between $\mathfrak{P}_G$ and the Bohr topology $\mathfrak{B}_G$ of a group $G$}
\label{P_G:section}

Recall that the finest precompact group topology on a group $G$ is called its \emph{Bohr topology}. 
We shall use the symbol $\mathfrak{B}_G$ for denoting the Bohr topology of a group $G$. The $\mathfrak{B}_G$-closure of the identity $1$ of $G$ is called the \emph{von Neumann kernel} of $G$; it is the kernel of a natural homomorphism of $G$ into its Bohr compactification $bG$.
The topology $\mathfrak{B}_G$ is Hausdorff if and only if $G$ admits a precompact Hausdorff group topology, or equivalently, if the discrete group $G$ is maximally almost periodic in the sense of von Neumann \cite{vonNeumann}.

The implication (i)$\to$(ii) of the following lemma is known
\cite{DT-Ischia}.

\begin{lemma}
\label{P_G:vs:B_G}
For an arbitrary group $G$, the following conditions are equivalent:
\begin{itemize}
\item[(i)]
$\mathfrak{B}_G$ 
is Hausdorff,
\item[(ii)]
 $\mathfrak{P}_G\subseteq \mathfrak{B}_G$.
\end{itemize}
\end{lemma}
\begin{proof}
(i)$\to$(ii)
Since $\mathfrak{B}_G$ is a precompact group topology on $G$ (being the finest precompact group topology on $G$) which is also Hausdorff by (i),
the definition of $\mathfrak{P}_G$ implies that every $\mathfrak{P}_G$-closed subset of $G$
 must be $ \mathfrak{B}_G$-closed.
 This establishes (ii).
 
\medskip
 (ii)$\to$(i)
Assume (ii), yet suppose that $\mathfrak{B}_G$ is not Hausdorff.
Then $G$ does not admit any precompact Hausdorff group topology, and by the definition of $\mathfrak{P}_G$ this means that 
\emph{every} subset of $G$ is $\mathfrak{P}_G$-closed in $G$; that is, $\mathfrak{P}_G$ is the discrete topology.
Combining this with the inclusion $\mathfrak{P}_G\subseteq \mathfrak{B}_G$ given by (ii), we conclude that $\mathfrak{B}_G$
must be discrete as well, in contradiction with our assumption that 
$\mathfrak{B}_G$ is not Hausdorff.
\end{proof}

\begin{remark}
\label{rem:6.2}
If $G$ is a finite group, then both $\mathfrak{P}_G$ and $\mathfrak{B}_G$ are discrete, and so
$\mathfrak{P}_G= \mathfrak{B}_G$.
\end{remark}

\begin{corollary}
\label{non-discretness:of:P_G}
For an infinite group $G$, the following conditions are equivalent:
\begin{itemize}
\item[(i)] $G$ is maximally almost periodic in its discrete topology, 
\item[(ii)] $\mathfrak{P}_G$ is non-discrete.
\end{itemize}
\end{corollary}
\begin{proof}
(i)$\to$(ii)
Let $G$ be an infinite group 
which is maximally almost periodic in its discrete topology. Then $\mathfrak{B}_G$ is
Hausdorff, and so $\mathfrak{P}_G\subseteq \mathfrak{B}_G$
by the implication (i)$\to$(ii) of Lemma \ref{P_G:vs:B_G}.
Since $G$ is infinite and $\mathfrak{B}_G$ is precompact, 
$\mathfrak{B}_G$ is non-discrete.
Therefore, the coarser topology $\mathfrak{P}_G$ cannot be discrete.

\medskip
(ii)$\to$(i) 
Since $\mathfrak{P}_G$ is non-discrete, $G$ must admit at least one precompact Hausdorff group topology, which implies (i). 
\end{proof}

The proof of the implication (ii)$\to$(i) of Corollary 
\ref{non-discretness:of:P_G}
does not require $G$ to be infinite.
On the other 
hand, 
the implication (i)$\to$(ii) 
requires 
$G$ 
to be 
infinite,
as every finite group $G$ is compact (and so maximally almost periodic) in its discrete topology, yet $\mathfrak{P}_G$ is discrete 
by Remark \ref{rem:6.2}.

\begin{corollary}
\label{non-discrete:precompact:topology:on:free:groups}
The following hold for a non-trivial free group $G$: 
\begin{itemize}
\item[(a)] 
the
Bohr topology $\mathfrak{B}_G$ of $G$ is Hausdorff;
\item[(b)] 
the
precompact Markov topology $\mathfrak{P}_G$ of $G$ is non-discrete.
\end{itemize}
\end{corollary}
\begin{proof}
Since $G$ is known to be residually finite, it admits a precompact Hausdorff group topology. Therefore, its finest Hausdorff precompact group topology $\mathfrak{B}_G$ is Hausdorff as well, showing (a). Item (a) implies that $G$ is maximally almost periodic in its discrete topology. Since $G$ is non-trivial (and thus, infinite),
 item (b) follows from the implication (i)$\to$(ii) of Corollary
\ref{non-discretness:of:P_G}.
\end{proof}

\begin{lemma}
\label{Bohr:compacts}
Let $H$ be a subgroup of a group $G$.
Then every 
$\mathfrak{B}_H$-compact subset of $H$ is
$\mathfrak{P}_G$-closed.
\end{lemma}
\begin{proof}
Let $K$ be a 
$\mathfrak{B}_H$-compact 
subset of $H$.
Let $\mathscr{T}$ be a precompact Hausdorff group topology on $G$. Then 
its restriction 
$\mathscr{T}\restriction_H$
to $H$ is a precompact group topology on $H$. 
Since the Bohr topology of $H$ 
is the finest precompact group topology on $H$,
$\mathscr{T}\restriction_H$ is coarser than 
$\mathfrak{B}_H$.
Since $K$ is a 
$\mathfrak{B}_H$-compact 
subset of $H$, it is also compact in the weaker topology $\mathscr{T}\restriction_H$, and thus also in the topology $\mathscr{T}$ on $G$.
Since $\mathscr{T}$ is Hausdorff, $K$ is $\mathscr{T}$-closed.

We have proved that $K$ is closed in every precompact Hausdorff group topology $\mathscr{T}$ on $G$. Therefore, $K$ is $\mathfrak{P}_G$-closed by the definition of $\mathfrak{P}_G$.
\end{proof}

\begin{corollary}
\label{corollary:when:H=G}
Every $\mathfrak{B}_G$-compact subset of a group $G$
is $\mathfrak{P}_G$-closed in it.
\end{corollary}

\begin{theorem}
\label{homeo:corollary}
Let $H$ be a subgroup of a group $G$, 
$K$ be a $\mathfrak{B}_H$-compact subset of $H$ and 
$\iota_K: (K,\mathfrak{B}_H\restriction_K)\to (K,\mathfrak{P}_G\restriction_K)$ be the identity map of $K$.
Then:
\begin{itemize}
\item[(i)]
$\iota_K$
is a closed map;
\item[(ii)]
if
$\mathfrak{B}_G$ is Hausdorff, then $\iota_K$ is a homeomorphism. 
\end{itemize}
\end{theorem}
\begin{proof}
(i)
Let $F$ be a closed subset of $(K,\mathfrak{B}_H\restriction_K)$.
Since $K$ is $\mathfrak{B}_H$-compact subset of $H$, so is its closed subset $F$. Applying Lemma \ref{Bohr:compacts} to $F$, we conclude that
$F$ is $\mathfrak{P}_G$-closed in $G$.
Since $F\subseteq K$, it follows that $F$ is $\mathfrak{P}_G\restriction_K$-closed in $K$.
We proved that the image $\iota_{K}(F)=F$ of every 
$\mathfrak{B}_H\restriction_K$-closed subset $F$ of $K$ under the 
map $\iota_{K}$ is $\mathfrak{P}_G\restriction_K$-closed in $K$, which implies (i).

\medskip
(ii)
Assume that $\mathfrak{B}_G$ is Hausdorff.
Then 
$\mathfrak{P}_G\subseteq \mathfrak{B}_G$
by the implication (i)$\to$(ii) of Lemma \ref{P_G:vs:B_G}.
Since $\mathfrak{B}_G$ is a precompact Hausdorff group topology on $G$, its restriction $\mathfrak{B}_G\restriction_H$ to the subgroup $H$ of $G$ is a precompact Hausdorff group topology on $H$, so $\mathfrak{B}_G\restriction_H\subseteq \mathfrak{B}_H$
by the maximality of the Bohr topology of $H$.
From $\mathfrak{P}_G\subseteq \mathfrak{B}_G$ we obtain also
$\mathfrak{P}_G\restriction_H\subseteq \mathfrak{B}_G\restriction_H$.
Thus,
$\mathfrak{P}_G\restriction_H\subseteq \mathfrak{B}_H$.
Since $K$ is a subset of $H$, it follows that
$\mathfrak{P}_G\restriction_K=(\mathfrak{P}_G\restriction_H)\restriction_K\subseteq \mathfrak{B}_H\restriction_K$.
This inclusion implies the continuity of the map $\iota_K$.
Since this map is also closed by (i), we conclude that 
$\iota_K$ is a homeomorphism.
\end{proof}

\begin{corollary}
\label{the:same:topologies:on:Bohr:compact:subsets}
Let $G$ be a maximally almost periodic group in its discrete topology. Then $\mathfrak{B}_G$ and $\mathfrak{P}_G$ induce the same (subspace) topology on every Bohr compact subset 
of $G$.
\end{corollary}
\begin{proof}
Since $G$ is maximally almost periodic in its discrete topology,
$\mathfrak{P}_G$ is Hausdorff. Now the conclusion follows from item (ii) of Theorem \ref{homeo:corollary}.
\end{proof}

\begin{corollary}
\label{sequences:in:P_G}
If $G$ is a non-commutative free group, then $(G,\mathfrak{P}_G)$ 
contains a countably infinite compact Hausdorff subspace with a single non-isolated point.
\end{corollary}
\begin{proof}
Let $G$ a non-commutative free group. Then $G=F(X)$
for a suitable set $X$ with  
at least two elements $x_0,x_1$.
Define $Y=\{x_0,x_1\}$ and $H=F(Y)$.
Now $H$ is a subgroup of $G$. By the result of Thom,
$H$ contains a 
faithfully indexed
sequence converging in its Bohr topology $\mathfrak{B}_H$ \cite[Proposition 3.7]{Thom}. 
Denote by $S$ this sequence taken together with its limit.
Then 
$S$ is a 
$\mathfrak{B}_H$-compact subset of $H$.
Since $\mathfrak{B}_G$ is Hausdorff by
Corollary \ref{non-discrete:precompact:topology:on:free:groups}(a), %
 the spaces $(S,\mathfrak{B}_H\restriction_S)$ and $(S,\mathfrak{P}_G\restriction_S)$ are homeomorphic
by 
item (ii) of Theorem \ref{homeo:corollary}.
Similarly,
$\mathfrak{B}_H$ is Hausdorff
by 
Corollary \ref{non-discrete:precompact:topology:on:free:groups}(a),
so 
the
restriction $\mathfrak{B}_H\restriction_S$
of $\mathfrak{B}_H$ 
to $S$
is Hausdorff as well.
Thus, $(S,\mathfrak{B}_H\restriction_S)$ is a compact Hausdorff space.
Since it is also a convergent sequence taken together with its limit, 
$(S,\mathfrak{B}_H\restriction_S)$ is 
a countably infinite compact Hausdorff space with a single non-isolated point.
\end{proof}

\noindent
{\bf Proof of Theorem \ref{precompact:Markov}:}
Let $G$ be a non-commutative free group.
Then the space $(G, \mathfrak{Z}_G)$ is Noetherian;
that is, each strictly decreasing sequence of its closed subsets stabilises
\cite[Corollary 5.2]{DT-Ischia}.
Therefore, all Hausdorff subspaces of $(G, \mathfrak{Z}_G)$ are finite.
On the other hand, 
$(G, \mathfrak{P}_G)$ contains a countably infinite Hausdorff subspace by Corollary \ref{sequences:in:P_G}.
Since $\mathfrak{Z}_G=\mathfrak{M}_G$ by Theorem
\ref{main:theorem},
it
follows that
$\mathfrak{M}_G\neq \mathfrak{P}_G$.
\qed

\begin{remark}
We note that Corollary \ref{corollary:when:H=G} does not offer anything really new for commutative groups $G$. Indeed, if $G$ is abelian, then $\mathfrak{B}_G$-compact sets are finite (\cite{G}; see also \cite[Corollary 6.4]{DS_Advances} for a more general result), and finite sets are $\mathfrak{P}_G$-closed because $\mathfrak{P}_G$ is a $T_1$-topology.
Similar comment applies to 
Corollary \ref{the:same:topologies:on:Bohr:compact:subsets}
(as well as 
Lemma \ref{Bohr:compacts} and 
Theorem \ref{homeo:corollary}), as finite subspaces of $T_1$-spaces are discrete.
\end{remark}

\section{Two new classes of groups defined by means of the precompact Markov topology}
\label{sec:11}

Inspired by the class $\mathcal{MZ}$, we denote by $\mathcal{PZ}$ the class of groups for which precompact Markov topology and Zariski topology coincide. Similarly,
we shall use the notation $\mathcal{PM}$ for denoting the class of groups for which precompact Markov topology and Markov topology coincide. Since 
$\mathfrak{Z}_G\subseteq \mathfrak{M}_G\subseteq \mathfrak{P}_G$ 
for every group $G$, 
the following inclusions hold:
\begin{equation}
\label{11:hj}
\mathcal{MZ}\supseteq \mathcal{PZ}\subseteq \mathcal{PM}.
\end{equation}

Our next remark shows that both inclusions in \eqref{11:hj} are proper and there are no other inclusions beyond the ones displayed.
\begin{remark}
\label{rem:11.1:h}
(i) Let $G=S(X)$ be an infinite symmetric group. Then $\mathfrak{Z}_G=\mathfrak{M}_G$ is the non-discrete topology 
of the pointwise convergence on $X$ \cite{BGP}, yet the topology 
$\mathfrak{P}_G$ is discrete \cite[Example 4.4(ii)]{DT-Ischia}. Therefore,
$G$ is in the class $\mathcal{MZ}$ but does not belong to the class $\mathcal{PM}$, and so is not in the smaller class $\mathcal{PZ}$ either. 

\smallskip
(ii) 
Let $G$ be 
the group 
from Example \ref{exa:Hesse}.
Then $\mathfrak{Z}_G$ is non-discrete yet $\mathfrak{M}_G$ is discrete. 
Since $\mathfrak{M}_G\subseteq \mathfrak{P}_G$, it follows that the topology $\mathfrak{P}_G$ is also discrete.
Therefore, $G$ is in the class $\mathcal{PM}$ but does not belong to the class $\mathcal{MZ}$, and so is not in the smaller class $\mathcal{PZ}$ either. 
\end{remark}

Item (i) of the following problem is a ``precompact version'' of Markov's problem on the description of (the groups in) the class $\mathcal{MZ}$.
\begin{problem}
\label{heir:of:markov}
\begin{itemize}
\item[(i)] Describe the class $\mathcal{PZ}$.
\item[(ii)] Describe the class $\mathcal{PM}$.
\end{itemize}
\end{problem}

\begin{remark}
\label{rem:6.3}
(i) It was proved in \cite{DS_MZ} that abelian groups belong to the class $\mathcal{PZ}$ (and thus, to the wider 
classes
$\mathcal{MZ}$
and
$\mathcal{PM}$). 

(ii) It was shown in \cite[Proposition 3.3]{DT-2012}
that the Heisenberg group $H$ over any field $K$ of characteristic 
$0$ is a nilpotent group of class $2$ 
such that $\mathfrak{M}_H$ is non-discrete while $\mathfrak{P}_H$ is discrete, so
$H$
does not belong to the class $\mathcal{PM}$ (and thus, does not belong to the narrower class $\mathcal{PZ}$ either).
\end{remark}

In connection with Remark \ref{rem:6.3}(ii)
and \eqref{11:hj},
it 
remains unclear if 
nilpotent groups 
belong to 
the
class
$\mathcal{MZ}$
(\cite[Question 12.1]{DS_MZ}; repeated in \cite[Question 8.13]{DT-Ischia} and \cite[Question 7.3]{DT-2012}).

\begin{question}
\label{question:2}
Can every group be represented as a quotient group of a group 
from the class $\mathcal{PZ}$, or at least a wider class $\mathcal{PM}$?
\end{question}

Since $\mathcal{PZ}\subseteq \mathcal{MZ}$ by 
\eqref{11:hj},
a positive answer to the $\mathcal{PZ}$-version of this question would strengthen
Corollary \ref{universal:quotient}. By Theorem \ref{precompact:Markov}, non-commutative free groups do not belong to the class $\mathcal{PM}$ (and hence to its subclass $\mathcal{PZ}$), so the idea of the proof of Corollary \ref{universal:quotient} cannot be used to answer Question \ref{question:2} positively. By the same reason, a negative answer to this question would imply that the correspondent class of groups is not closed under taking quotients.

In connection with Theorem \ref{every:group:is:a:quotient:of:MZ:group} and Corollary \ref{subgroup:corollary}, the following question
arises:
\begin{question}
\begin{itemize}
\item[(i)]
Can every group be represented as a \emph{normal} subgroup of a group from the class $\mathcal{MZ}$?
\item[(ii)]
Does every \emph{normal} subgroup of a group in the class $\mathcal{MZ}$ belong to the class $\mathcal{MZ}$?
\end{itemize}
\end{question}
Remark \ref{rem:11.1:h}(ii) gives an
example of
a group outside the class 
$\mathcal{MZ}$,
so
positive answers to both items of this question are mutually exclusive.

\begin{question}
\label{question:3}
Can every group be represented as a (normal) subgroup of a group 
from the class $\mathcal{PZ}$, or at least a wider class $\mathcal{PM}$?
\end{question}
Since $\mathcal{PZ}\subseteq \mathcal{MZ}$ by 
\eqref{11:hj},
a positive answer to the $\mathcal{PZ}$-version of this question would strengthen
Theorem \ref{every:group:is:a:quotient:of:MZ:group}. 
Since symmetric 
groups do not belong to the class $\mathcal{PM}$ (and hence to its subclass $\mathcal{PZ}$)
by Remark \ref{rem:11.1:h}(i),
 the idea of the proof of 
Theorem \ref{every:group:is:a:quotient:of:MZ:group}
cannot be used to answer Question \ref{question:3} positively. 
By the same reason, a negative answer to this question would imply that the correspondent class of groups is not closed under taking (normal) subgroups.

\bigskip
\noindent
{\bf Acknowledgements:\/}
We are grateful to Professor
Kazuhiro Kawamura for a stimulating question resolved in Corollary \ref{corollary:1.4}.

\end{document}